\documentclass[12pt]{article}

\pdfoutput=1

\usepackage{amsmath, amssymb, amsthm, enumerate, fullpage, bm}
\usepackage{verbatim, enumitem, bbm, etoolbox}
\usepackage{latexsym}

\apptocmd{\lim}{\limits}{}{}

\theoremstyle{definition}
\newtheorem{thm}{Theorem}
\newtheorem{theorem}[thm]{Theorem}

\newtheorem{lemma}[thm]{Lemma}

\numberwithin{thm}{section}

\numberwithin{subcase}{case}

\theoremstyle{definition}
\newtheorem{definition}[thm]{Definition}

\def\forkindep{\mathrel{\raise0.2ex\hbox{\ooalign{\hidewidth$\vert$\hidewidth\cr\raise-0.9ex\hbox{$\smile$}}}}}

\def\Ind{\setbox0=\hbox{$x$}\kern\wd0\hbox to 0pt{\hss$\mid$\hss}
	\lower.9\ht0\hbox to 0pt{\hss$\smile$\hss}\kern\wd0}

\def\Notind{\setbox0=\hbox{$x$}\kern\wd0\hbox to 0pt{\mathchardef
		\nn=12854\hss$\nn$\kern1.4\wd0\hss}\hbox to 0pt{\hss$\mid$\hss}\lower.9\ht0
	\hbox to 0pt{\hss$\smile$\hss}\kern\wd0}

\def\phi{\varphi}

\def\<{\langle}
\def\>{\rangle}

\begin{document}	

	\bibliographystyle{plain}
	
	\author{Douglas Ulrich\!\!\
	\thanks{Partially supported
by Laskowski's NSF grant DMS-1308546.}\\
Department of Mathematics\\University of Maryland}
	\title{Low is a Dividing Line in Keisler's Order}
	\date{\today} 
	
	\maketitle
	
\begin{abstract}
We show in $ZFC$ that the class of low theories forms a dividing line in Keisler's order. That is, if $T$ is low and $T' \trianglelefteq T$ then $T'$ is low. We also show there is a minimal nonlow theory $T_{cas}$. 
\end{abstract}

\section{Introduction}

Keisler's order $\trianglelefteq$ is a partial order on countable complete theories, introduced by Keisler in \cite{Keisler}. He began with the observation that if $\mathcal{U}$ is a $\lambda$-regular ultrafilter on $\mathcal{P}(\lambda)$, if $T$ is a complete countable theory, and if $M \models T$, then whether or not $M^\lambda/\mathcal{U}$ is $\lambda^+$-saturated only depends on $T$. If this is the case we say that $\mathcal{U}$ $ \lambda^+$-\emph{saturates} $T$, and we set $T_1 \trianglelefteq T_2$ if whenever $\mathcal{U}$ is a regular ultrafilter on $\mathcal{P}(\lambda)$, if $\mathcal{U}$ $\lambda^+$-saturates $T_2$ then $\mathcal{U}$ $\lambda^+$-saturates $T_1$.  Recently, Malliaris and Shelah \cite{DividingLine} have introduced the method of \emph{Separation of Variables}, which allows us to translate the problem of Keisler's order to the construction of ultrafilters on arbitrary complete Boolean algebras, and have gone on to find many applications. In particular, in \cite{Optimals} they prove that if there is a supercompact cardinal, then simplicity is a dividing line in Keisler's order.

Say that the complete theory $T$ is \emph{low} if it is simple and for every formula $\phi(x, \overline{y})$, there is some $k$ such that for all $\overline{b}$, if $\phi(x, \overline{b})$ does not $k$-divide over the empty set then it does not divide over the empty set.\footnote{This is the standard definition of low, for instance it is equivalent to the original definition in \cite{Buechler}; Malliaris has defined low slightly differently in \cite{MalliarisFlex}, namely not requiring $T$ to be simple. } In this paper we are able to combine the argument in \cite{Optimals} with the partial definability of forking in low theories to get that the class of low theories is a dividing line in Keisler's order, using just $ZFC$.  By a \emph{dividing line} we mean a $\trianglelefteq$-downward closed set of complete first order theories; so in other words, if $T$ is low and $T' \trianglelefteq T$ then $T'$ is low.

In section~\ref{LowSection} we prove a useful lemma about low theories. In Section~\ref{SetupSection} we review the general setup of \cite{DividingLine}, in particular stating the Existence Theorem and  Separation of Variables. In Section~\ref{TypeAmalg} we introduce a simplified relative of $(\lambda, \mu, \theta, \sigma)$-explicitly simple from \cite{Optimals}, namely having $(\lambda, \mu, \sigma)$-type amalgamation. In Section~\ref{LowDividingLine} we show that the class of low theories forms a dividing line in Keisler's order. In Section~\ref{MinNonLowSec} we show there is a minimal nonlow theory $T_{cas}$ in Keisler's order.

\section{Low Theories}\label{LowSection}
Recall that by a theorem of Kim \cite{KimForking}, in any simple theory $T$, forking is the same as dividing; that is $\phi(x, \overline{a})$ forks over $A$ if and only if it divides over $A$. We thus use the terms \emph{forking} and \emph{dividing} interchangeably.

We give the following equivalence for $T$ being low. (C) is Buechler's original definition of lowness from \cite{Buechler}; equivalently it states that for every formula $\phi(x, \overline{y})$, $D(x=x, \phi(x, \overline{y})) < \omega$, where $D$ is the $D$-rank for low theories; in the same paper he proved the equivalence of that with our definition in terms of dividing. Thus (A) if and only if (C) is already known. 

\begin{theorem}\label{LowEquiv}
Suppose $T$ is simple. Then the following are equivalent:

\begin{itemize}
\item[(A)] $T$ is low.
\item[(B)] Suppose $\phi(x, \overline{b})$ does not fork over $A$. Then there is some $\overline{c} \in A$ and some $\psi(\overline{y}, \overline{z}) \in \mbox{tp}(\overline{b}, \overline{c})$ such that whenever $(\overline{b}', \overline{c}') \models \psi(\overline{y}, \overline{z})$, then $\phi(x, \overline{b}')$ does not fork over $\overline{c}'$.
\item[(C)] For every formula $\phi(x, \overline{y})$, there is some $k$ such that there is no sequence $(\overline{b}_i: i < k)$ such that $\bigwedge_{i < k} \phi(x, \overline{b}_i)$ is consistent, and such that for each $i < k$, $\phi(x, \overline{b}_i)$ forks over $\{\overline{b}_j: j < i\}$. 
\end{itemize}
\end{theorem}
\begin{proof}

(A) implies (B): Choose $k$ such that if $\phi(x, \overline{b}')$ does not $k$-divide over $\emptyset$ then it does not divide over $\emptyset$. It follows that if $A'$ is any set and $\phi(x, \overline{b}')$ does not $k$-divide over $A'$ then $\phi(x, \overline{b}')$ does not divide over $A'$. Since $\phi(x, \overline{b})$ does not divide over $A$, $\phi(x, \overline{b})$ does not $k$-divide over $A$; by a compactness argument we can choose $\overline{c} \in A$ and $\psi(\overline{y}, \overline{z}) \in \mbox{tp}(\overline{b}, \overline{c})$ such that whenever $\models \psi(\overline{b}', \overline{c}')$ then $\phi(x, \overline{b}')$ does not $k$-divide over $\overline{c}'$. But then by choice of $k$, $\phi(x, \overline{b}')$ does not divide over $\overline{c}'$.

(B) implies (C): Suppose (C) holds, and let $\phi(x, \overline{y})$ be given. Let $\Gamma$ be the partial type in the variables $(\overline{y}_\alpha: \alpha < \omega_1)$ asserting: 
\begin{itemize}
\item For each $s \in [\omega_1]^{<\omega}$, $\exists x \bigwedge_{\alpha \in s} \phi(x, \overline{y}_\alpha)$; 
\item For each $\alpha < \omega_1$, $\phi(x, \overline{y}_\alpha)$ forks over $(\overline{y}_{\beta}: \beta < \alpha)$. 
\end{itemize} 
The second item is possible to express by hypothesis.

Suppose towards a contradiction that $\Gamma$ were consistent; choose $(\overline{b}_\alpha: \alpha < \omega_1)$ a realization of $\Gamma$. Let $p(x)$ be the type over $(\overline{b}_\alpha: \alpha < \omega_1)$ asserting that $\phi(x, \overline{b}_\alpha)$ holds for each $\alpha < \omega_1$. Then $p(x)$ is consistent but forks over every countable subset of its domain, contradicting simplicity of $T$. 

Thus $\Gamma$ is inconsistent; by symmetry we can choose $n$ such that $\Gamma \restriction_{(\overline{y}_i: i < n)}$ is inconsistent. This just says that (C) holds.

(C) implies (A): let $\phi(x, \overline{y})$ be given, and let $k$ be as in (C). We claim that if $\phi(x, \overline{b})$ does not $k+1$-divide over $\emptyset$ then $\phi(x, \overline{b})$ does not divide over $\emptyset$. Indeed, suppose towards a contradiction that we had an indiscernible sequence $(\overline{b}_i: i < \omega)$ such that $\phi(x, \overline{b}_i)$ were $k$-consistent but $k+1$-inconsistent. Then $(\overline{b}_i: i < k)$ is a counterexample to the choice of $k$.

\end{proof}

\section{Existence Theorem and Separation of Variables}\label{SetupSection}

We recall the general setup introduced in \cite{DividingLine}, \cite{InfManyClass} and \cite{Optimals}. Given $\mathcal{B}$ a complete Boolean algebra and an implicit cardinal $\lambda$, a \emph{distribution} on $\mathcal{B}$ is a sequence $(\mathbf{a}_s: s \in [\lambda]^{<\aleph_0})$ of nonzero elements of $\mathcal{B}$, such that $\mathbf{a}_\emptyset = 1$, and for $s \subseteq t$, $\mathbf{a}_s \geq \mathbf{a}_t$. When context makes $\lambda$ clear, we just write $(\mathbf{a}_s)$. The distribution $(\mathbf{a}_s: s \in [\lambda]^{<\aleph_0})$ is \emph{multiplicative} if for each $s \in [\lambda]^{<\aleph_0}$, $\mathbf{a}_s = \bigwedge_{\alpha \in s} \mathbf{a}_{\{\alpha\}}$. $(\mathbf{a}_s)$ \emph{refines} $(\mathbf{b}_s)$ if each $\mathbf{a}_s \leq \mathbf{b}_s$. If $\mathcal{D}$ is a filter on $\mathcal{B}$, then $(\mathbf{a}_s)$ is \emph{in} $\mathcal{D}$ if each $\mathbf{a}_s \in \mathcal{D}$. 

For example, suppose $\mathcal{U}$ is an ultrafilter on $\mathcal{P}(I)$ for some index set $I$, and $M$ is a model with ultrapower $\overline{M} := M^I/\mathcal{U}$, and $p(x) = \{\phi_\alpha(x, \overline{a}_\alpha): \alpha < \lambda\}$ is a partial type over $\overline{M}$. Then we can form the distribution $(\mathbf{a}_s: s \in [\lambda]^{<\aleph_0})$, where we let $\mathbf{a}_s := \{i \in I: M[i] \models \exists x \bigwedge_{\alpha \in s} \phi_\alpha(x, \overline{a}_\alpha[i])\}$ (where $M[i]$ is the $i$'th copy of $M$, and where we are fixing some lifting $i \mapsto \overline{a}_\alpha[i]$ of each $\overline{a}_\alpha$). Then $(\mathbf{a}_s)$ is a distribution in $\mathcal{U}$. The following is a standard fact, see for instance \cite{ConsUltrPOV} Observation 1.8: if $\mathcal{U}$ is $\lambda$-regular, then $p(x)$ is realized in $\overline{M}$ if and only if $(\mathbf{a}_s)$ has a multiplicative refinement in $\mathcal{U}$.

Also, given $\mu \geq \theta$ with $\theta$ regular and $\mu = \mu^{<\theta}$, and given a set $X$, let $P_{X, \mu, \theta}$ be the set of all partial functions from $X$ to $\mu$ of cardinality less than $\theta$, ordered by reverse inclusion; let $\mathcal{B}_{X, \mu, \theta}$ be the Boolean algebra completion of $P_{X, \mu, \theta}$. So $\mathcal{B}_{X, \mu, \theta}$ has the $\mu^+$-c.c. and is $<\theta$-distributive. For each $f \in P_{X, \mu, \theta}$ let $\mathbf{x}_f$ be the corresponding element of $\mathcal{B}_{X, \mu, \theta}$.

Let $\mathcal{B}$ be a complete Boolean algebra, let $T$ be a complete countable theory, let $\lambda$ be a cardinal and let $\mathcal{U}$ be an ultrafilter on $\mathcal{B}$. Then $\mathcal{U}$ is $(\lambda, \mathcal{B}, T)$-\emph{moral} if whenever $(\mathbf{a}_s: s \in [\lambda]^{<\aleph_0})$ is a certain kind of distribution from $\mathcal{U}$ (namely, a $(\mathcal{B}, T, \overline{\phi})$-possibility for some sequence of formulas $\overline{\phi}$) then $(\mathbf{a}_s)$ has a multiplicative refinement in $\mathcal{U}$ (for more details see section 6 of \cite{DividingLine}; we won't need to work with the definition directly.) This notion is actually a generalization of $\lambda^+$-saturation. That is, suppose $\mathcal{B} = \mathcal{P}(\lambda)$ and $\mathcal{U}$ is $\lambda$-regular. Then $\mathcal{U}$ $\lambda^+$-saturates $T$ if and only if $\mathcal{U}$ is $(\lambda, \mathcal{B}, T)$-moral.

We now quote some key facts from \cite{DividingLine}. (By Theorem 5.2 of \cite{DividingLine}, excellence is the same as goodness, so we phrase everything in terms of goodness.) The first fact is the \emph{Existence Theorem}, which is Theorem 8.1 of \cite{DividingLine}:

\begin{theorem}\label{ExistenceTheorem}
Suppose $\mu$ is a cardinal; write $\lambda = \mu^+$. Suppose $\mathcal{B}$ is a complete Boolean algebra with the $\mu^+$-c.c. Then there exists a regular good filter $\mathcal{D}$ on $\mathcal{P}(\lambda)$ such that $\mathcal{P}(\lambda)/\mathcal{D} \cong \mathcal{B}$.
\end{theorem}

Actually, with a little cardinal arithmetic one can show this holds for any complete Boolean algebra $\mathcal{B}$ with $|\mathcal{B}| \leq 2^\lambda$.

The other fact we need from \cite{DividingLine} is the following Theorem 6.13, termed \emph{Separation of Variables}.

\begin{theorem}\label{SeparationOfVariablesTheorem}
Suppose $\mathcal{B}$ is a complete Boolean algebra, and $\mathcal{D}$ is a regular good filter on $\mathcal{P}(\kappa)$ for some $\kappa$, and $\mathbf{j}: \mathcal{P}(\kappa)/\mathcal{D} \cong \mathcal{B}$ is a Boolean algebra isomorphism. Suppose $\mathcal{U}$ is any ultrafilter on $\mathcal{B}$; write $\mathcal{U}_* := \mathbf{j}^{-1}(\mathcal{U})$. Then for any complete theory $T$, $\mathcal{U}$ is $(\lambda, \mathcal{B}, T)$-moral if and only if $\mathcal{U}_*$ $\lambda^+$-saturates $T$.
\end{theorem}

\section{$(\lambda, \mu, \sigma)$-type amalgamation}\label{TypeAmalg}
In \cite{Optimals}, a property of theories called $(\lambda,\mu, \theta, \sigma)$-explicit simplicity is defined. In this section we present a streamlined version, namely $(\lambda, \mu, \sigma)$-type amalgamation. For the most part we have just made notational simplifications, such as identifying $\theta = \sigma$, and getting rid of the ambient model $M$ and type $p(x)$ over $M$. There are two changes we have made not of this nature. First, in the definition of good instantiation, we have required a stronger independence property; this is to arrange that Lemma~\ref{InstantiationsAreInd} is true. Second, we allow our colorings $G$ to take values in the partial order $P_{\lambda, \mu, \sigma}$ rather than the set $\mu$.

\begin{definition}
Suppose $X$ is a set and $\kappa$ is a cardinal. $\mbox{cl}: \mathcal{P}(X) \to \mathcal{P}(X)$ is a $\kappa$-\emph{closure-relation} if there are $\gamma_* < \kappa$ and functions $\overline{F} = (F_\gamma: \gamma < \gamma_*)$ such that each $F_\gamma: X^{<\omega} \to X$, and such that for every $A \subseteq X$, $\mbox{cl}(A)$ is the least subset of $X$ closed under $\overline{F}$. We write $\mbox{cl} = \mbox{cl}_{\overline{F}}$.

If $\mbox{cl}_0, \mbox{cl}_1$ are $\kappa$-closure relations on $X$, then say that $\mbox{cl}_1$ \emph{expands} $\mbox{cl}_0$ if $\mbox{cl}_0(A) \subseteq \mbox{cl}_1(A)$ for all $A \subseteq X$. 
\end{definition}

Choose distinct variables $x_\alpha$ for every ordinal $\alpha$. If $w$ is a set of ordinals let $\overline{x}_w = (x_\alpha: \alpha \in w)$ be the sequence of variables indexed by $w$ in increasing order.

\begin{definition}
Let $T$ be a countable simple theory.  Suppose $\aleph_0 < \sigma \leq \lambda$, with $\sigma$ regular. Suppose $\mbox{cl}$ is a $\sigma$-closure relation on $\lambda$. Then let $\mathcal{R}_{\lambda, \sigma, T}(\mbox{cl})$, or $\mathcal{R}(\mbox{cl})$ when $\lambda, \sigma, T$ are understood, be the set of all triple $(w, q, p)$ where:

\begin{itemize}
\item $w \in [\lambda]^{<\sigma}$ is $\mbox{cl}$-closed;
\item $q(\overline{x}_w)$ is a complete type in the variables $\overline{x}_w$;
\item $p(x, \overline{x}_w)$ is a complete type in the variables $(x, \overline{x}_w)$ extending $q(\overline{x}_w)$;
\item Suppose $\overline{b} = (b_\alpha: \alpha \in w)$ is a realization of $q(\overline{x}_w)$ in $\mathfrak{C}_T$, and $p(x, \overline{b})$ is the corresponding type over $\overline{b}$. Then:
\begin{itemize}
\item For each closed $w' \subseteq w$, $\overline{b}_{w'} \preceq \mathfrak{C}_T$ (that is $\overline{b}_{w'}$ enumerates an elementary substructure of $\mathfrak{C}_T$)

\item $p(x, \overline{b})$ does not fork over $\overline{b} \restriction_{\mbox{cl}(\emptyset)}$.
\end{itemize}
\end{itemize}

\end{definition}

As a silly example, note if $\mbox{cl}(\emptyset)$ is finite then $\mathcal{R}(\mbox{cl}) = \emptyset$ since there are no finite elementary substructures of $\mathfrak{C}_T$.

\begin{definition}
Suppose $((w_i, q_i, p_i): i < i_*)$ is a finite sequence from $\mathcal{R}(\mbox{cl})$. Write $w = \bigcup_i w_i$ and suppose $\overline{b} = (b_\alpha: \alpha \in w)$ is a sequence in $\mathfrak{C}_T$. Then $\overline{b}$ is a \emph{good instantiation} of $(\overline{w}, \overline{q}, \overline{p})$ if:

\begin{itemize}
\item For each $i < i_*$, $\overline{b}_{w_\gamma}$ is a realization of $q_\gamma$;
\item For each $\alpha \in w$, if we set $v = \bigcap\{w_i: i < i_*, \, \alpha \in w_i\}$ then $tp\left( b_\alpha / \overline{b}_{w \cap \alpha}\right)$ does not fork over $\overline{b}_{v \cap \alpha}$;
\item For each $i, j < i_*$, $p_i(x, \overline{b}_{w_{i}})$ and $p_{j}(x, \overline{b}_{w_{j}})$ agree on $\overline{b}_{w_{i} \cap w_{j}}$.
\end{itemize}

\end{definition}

These should be viewed as type-amalgamation problems. We make the following definition.

\begin{definition} Suppose $(P, <)$ is a lower semilattice (a partial order with meets and $0$). Then $(M_u, p_u: u \in P)$ is an \emph{independent system of types} if

\begin{itemize}
\item Each $M_u$ is an elementary submodel of $\mathfrak{C}$, and for all $u \leq v \in P$, $M_u \subseteq M_v$, and for all $(u_i: i < n)$, $(v_j: j < m)$, we have that

$$\bigcup_{i} M_{u_i} \forkindep_{\bigcup_{i ,j} M_{u_i \wedge v_j}} \bigcup_{j } M_{v_j}.$$ 

In particular each $M_u \cap M_v =M_{u \cap v}$.

\item Each $p_u(x)$ is a complete type over $M_u$ that does not fork over $M_0$, and $(p_u(x): u \in P)$ are pairwise compatible.
\end{itemize}
\end{definition}

Given an independent system of types $(M_u, p_u(x): u \in P)$, we are interested in asking: when is $\bigcup_u p_u(x)$ consistent and nonforking over $M_0$?

\begin{lemma}\label{InstantiationsAreInd}
Suppose $((w_i, q_i, p_i): i < i_*)$ is a finite sequence from $\mathcal{R}(\mbox{cl})$ and $\overline{b}$ is a good instantiation of $(\overline{w}, \overline{q}, \overline{p})$. Let $P$ be the closure of $\{w_i: i < n\}$ under intersections, ordered by subset. For each $u \in P$ let $p_u(x, \overline{b}_u)$ be $p_i(x, \overline{b}_{w_i})$ for some or any $w_i \supseteq u$. Then $(\overline{b}_{u}, p_u(x, \overline{b}_{u}): u  \in P)$ is an independent system of types.
\end{lemma}
\begin{proof}
Write $w= \bigcup_i w_i$.

Let $u_i: i <n$, $v_j: j < m$ be given elements of $P$; write $u_* = \bigcup_{i < n} u_i$ and let $v_* = \bigcup_{j < m} v_j$. Note that $\bigcup_{i, j} \overline{b}_{u_i \wedge v_j} = \bigcup_{i, j} \overline{b}_{u_i \cap v_j} = \overline{b}_{u_* \cap v_*}$, so we want to show that $\overline{b}_{u_*} \forkindep_{\overline{b}_{u_* \cap v_*}} \overline{b}_{v_*}$. We show by induction on $\alpha < \lambda$ that $\overline{b}_{u_* \cap \alpha} \forkindep_{\overline{b}_{u_* \cap v_* \cap \alpha}} \overline{b}_{v_* \cap \alpha}$. $\alpha = 0$ and $\alpha$ limit are both trivial. Suppose $\alpha = \beta + 1$; if $\beta \not \in u_* \cup v_*$ or $\beta \in u_* \cap v_*$ this step is trivial, so we can assume $\beta \in u_* \backslash v_*$. We can suppose after renumbering that $\beta \in u_i$ iff $i < n'$, for some $0 < n' \leq n$.

By definition of a good instantiation, if we set $u = \bigcap\{w_i: \beta \in w_i\}$ then $b_\beta \forkindep_{\overline{b}_{u \cap \beta}} \overline{b}_{w \cap \beta}$. Note $u \subseteq u_i$ for each $i < n'$. Thus $u \subseteq u_*$, so $b_\beta \forkindep_{\overline{b}_{u_* \cap \beta}} \overline{b}_{w \cap \beta}$, so $\overline{b}_{u_* \cap \alpha} \forkindep_{\overline{b}_{u_* \cap \beta}} \overline{b}_{v_* \cap \beta}$. Thus, by transitivity of forking and the inductive hypothesis, $\overline{b}_{u_* \cap \alpha} \forkindep_{\overline{b}_{u_* \cap v_* \cap \beta}} \overline{b}_{v_* \cap \beta}$; since $\beta \not \in v_*$ this is what we wanted to prove.
\end{proof}

Recall that $P_{X, \mu, \sigma}$ is the set of all partial functions from $X$ to $\mu$ of cardinality less than $\sigma$, ordered by reverse inclusion.

\begin{definition}
Suppose $\aleph_0 < \sigma \leq \mu = \mu^{<\sigma} \leq \lambda$ with $\sigma$ regular, $\mbox{cl}$ is a $\sigma$-closure relation on $\lambda$ and $T$ is a countable simple theory. Then say that $T$ has $(\lambda, \mu, \sigma, \mbox{cl})$-\emph{type amalgamation} if there is some $G: \mathcal{R}(\mbox{cl}) \to P_{\lambda, \mu, \sigma}$ satisfying the following. Suppose $(\overline{w}, \overline{q}, \overline{p}) = ((w_i, q_i, p_i): i < i_*)$ is a finite sequence from $\mathcal{R}(\mbox{cl})$ such that $\bigcup_{i < i_*} G(w_i, q_i, p_i)$ is a function. Then for every good instantiation $\overline{b}$ of $(\overline{w}, \overline{q}, \overline{p})$, we have that $\bigcup_{i < i_*} p_i(x, \overline{b}_{w_i})$ is a consistent partial type which does not fork over $\overline{b}_{\mbox{cl}(\emptyset)}$.
\end{definition}

\begin{definition}
Suppose $\aleph_0 < \sigma \leq \mu = \mu^{<\sigma} \leq \lambda$ with $\sigma$ regular, and $T$ is a countable simple theory. Then say that $T$ has $(\lambda, \mu, \sigma)$-\emph{type amalgamation} if $T$ has $(\lambda, \mu, \sigma, \mbox{cl})$-type amalgamation for every sufficiently expanded $\sigma$-closure operation $\mbox{cl}$ on $\lambda$.
\end{definition}

We now show that simple theories have $(\mu^+, \mu, \sigma)$-type-amalgamation for all $\sigma, \mu$ as above. Towards this we prove two lemmas; the first is just Claim 4.7 (2) from \cite{Optimals}.

\begin{lemma}\label{SimpleIsExplicitLemma1}
Suppose $\aleph_0 < \sigma \leq \mu = \mu^{<\sigma}$ with $\sigma$ regular; write $\lambda = \mu^+$.
There is a $\sigma$-closure operation $\mbox{cl}$ on $\lambda$ and a coloring $G: [\lambda]^{<\sigma} \to \mu$ such that whenever $(w_i: i < i_*)$ is a finite monochromatic sequence from $[\lambda]^{<\sigma}$, with each $w_i$ $\mbox{cl}$-closed, then the closure of $\overline{w}$ under intersections forms a tree under subset. (This then also holds for any expansion of $\mbox{cl}$.)
\end{lemma}

\begin{proof}
For each $\alpha < \lambda$ choose $f_\alpha: \alpha \to |\alpha|$ a bijection; so always $f_\alpha: \alpha \to \mu$. For $\beta < \alpha$, define $F_0(\alpha, \beta) = f_\alpha(\beta)$; and for $\beta < |\alpha|$, define $F_1(\alpha, \beta) = f_\alpha^{-1}(\beta)$. Let $\mbox{cl}$ be generated by $F_0$ and $F_1$. Define $G: [\lambda]^{<\sigma} \to \mu$ so that $G(w) = G(w')$ iff $w\cap\mu = w' \cap \mu$. We claim that this works.

Note that if $G(w_0) = G(w_1) = \ldots = G(w_n)$ then these are also equal to $G(w_0 \cap \ldots \cap w_n)$. Suppose we are given $(w_i: i < i_*)$. Let $P$ be the closure of $(w_i: i < i_*)$ under intersections; so $P$ is also monochromatic. It suffices to show that whenever $w_0 \subseteq w_1$ are both from $P$,  we have that $w_0$ is an initial segment of $w_1$. Suppose $\alpha \in w_0$ and $\beta < \alpha$ is in $w_1$. Write $\gamma = f_\alpha(\beta)$; then $\gamma < \mu$ so $\gamma \in w_1 \cap \mu = w_0 \cap \mu$. So $f_\alpha^{-1}(\gamma) = \beta \in w_0$. 
\end{proof}

\begin{lemma}\label{SimpleIsExplicitLemma2}
Suppose $T$ is a simple theory and $(S, <)$ is a tree and $(M_s, p_s(x): s \in S)$ is an independent system of types. Then $\bigcup_s p_s(x)$ is consistent and does not fork over $M_0$.
\end{lemma}
\begin{proof}
We show by induction on $s \in S$ that $\bigcup_{t \geq s} p_t(x)$ is consistent and does not fork over $M_0$. If $s$ is a leaf this is obvious. Suppose $s \in S$ and we have proved the claim for each $t > s$. Let $s_i: i < k$ list the immediate successors of $s$, and for each $i < s$ let $A_i = \bigcup_{t \geq s_i} M_t$. By definition of independent systems of types, we have that $(A_i: i < k)$ is independent over $M_s$. For each $i < k$ let $p_i(x) = \bigcup_{t \geq s_i} p_t(x)$. Then each $p_i(x)$ is a consistent partial type over $A_i$ which is complete over $M_s$ and does not fork over $M_0$. By the independence theorem, $\bigcup_{i < k} p_i(x)$ is consistent and does not fork over $M_s$. Write $A= \bigcup_{i < k} A_i$ and choose $a$ realizing $\bigcup_{i < k} p_i(x)$ such that $tp(a/A)$ dnf over $M_s$. Then $tp(a /M_s) = p_s(x)$ does not fork over $M_0$, so we conclude by transitivity that $\bigcup_{i < k} p_i(x)$ dnf over $M_0$.
\end{proof}

\begin{theorem}\label{SimpleIsExplicit}
Suppose $\aleph_0 < \sigma \leq \mu = \mu^{<\sigma} $ with $\sigma$ regular. Suppose $T$ is a countable simple theory. Then $T$ has $(\mu^+, \mu, \sigma)$-type amalgamation.
\end{theorem}
\begin{proof}
Clear, by the preceding lemmas.
\end{proof}

\section{Low is a Dividing Line}\label{LowDividingLine}

In \cite{Optimals}, \cite{InfManyClass} the notion of a perfect ultrafilter is introduced. We give an equivalent definition. Write $\mathcal{B} = \mathcal{B}_{2^\lambda, \mu, \sigma}$ and for each $\alpha < 2^\lambda$ write $\mathcal{B} =\mathcal{B}_{\alpha, \mu, \sigma}$. Also, if $\mathcal{U}$ is an ultrafilter on $\mathcal{B}$ and $\alpha < 2^\lambda$ then let $\mathcal{U} \restriction_{\mathcal{B}_\alpha}$ be the filter on $\mathcal{B}$ generated by $\mathcal{U} \cap \mathcal{B}_\alpha$.

\begin{definition}
The ultrafilter $\mathcal{U}$ on $\mathcal{B}$ is $\lambda$-\emph{perfect} if whenever $(\mathbf{b}_s: s \in [\lambda]^{<\aleph_0})$ is a distribution in $\mathcal{U}$, (A) implies (B):

\begin{itemize}
\item[(A)] For every $\delta < 2^\lambda$ (or equivalently, for arbitrarily large $\delta < 2^\lambda$) there is a multiplicative refinement $(\mathbf{b}'_s: s \in [\lambda]^{<\aleph_0})$ of $(\mathbf{b}_s)$ such that each $\mathbf{b}'_s$ is nonzero mod $\mathcal{U} \restriction_{\mathcal{B}_\delta}$;
\item[(B)] $(\mathbf{b}_s)$ has a multiplicative refinement in $\mathcal{U}$.
\end{itemize}
\end{definition}

The argument for the following theorem mirrors the saturation argument from \cite{Optimals} Theorem 7.3, where it is shown that if $\mathcal{U}$ is a certain kind of $\aleph_1$-complete ultrafilter on $\mathcal{B}_{2^{\sigma^+}, \sigma, \sigma}$, for $\sigma$ a supercompact cardinal, then $\mathcal{U}$ is $(\lambda, \mathcal{B}, T)$-moral for every simple theory $T$. The key difference between that argument and the following proof is that we have replaced the appeal to $\aleph_1$-completeness of $\mathcal{U}$ by an appeal to lowness of $T$.

\begin{theorem}\label{SatArg}
Suppose $\aleph_0 < \sigma \leq \mu = \mu^{<\sigma} \leq \lambda$ with $\sigma$ regular; write $\mathcal{B} = \mathcal{B}_{2^\lambda, \mu, \sigma}$ and for each $\alpha < \lambda$ write $\mathcal{B}_\alpha = \mathcal{B}_{\alpha, \mu, \sigma}$. Suppose $T$ has $(\lambda, \mu, \sigma)$-type amalgamation and is low, and $\mathcal{U}$ is a $\lambda$-perfect ultrafilter on $\mathcal{B}$. Then $\mathcal{U}$ is $(\lambda, \mathcal{B}, T)$-moral.
\end{theorem}
\begin{proof}
Let $I$ be an index set of size $\lambda$, and choose a regular good filter $\mathcal{D}_0$ on $\mathcal{P}(I)$ and an isomorphism $\mathbf{j}: \mathcal{P}(I)/\mathcal{D}_0 \cong \mathcal{B}$ (this is possible by the Existence Theorem). Write $\mathcal{U}_* = \mathbf{j}^{-1}(\mathcal{U})$. We want to show that $\mathcal{U}_*$ $\lambda^+$-saturates $T$.

Let $M \models T$, and let $\overline{M} = M^I/\mathcal{D}$; we want to show that $\overline{M}$ is $\lambda^+$-saturated. Choose $M_0 \preceq \overline{M}$ an elementary substructure of size $\lambda$ and choose $p(x)$ a complete type over $M_0$. We want to show $p(x)$ is realized in $\overline{M}$.

Choose $\overline{a} = (a_\gamma: \gamma < \lambda)$ a sequence from $M^I$, such that $M_0 = \{[[a_\gamma/\mathcal{U}_*]]: \gamma < \lambda\}$, and further for every formula $\phi(x, x_0, \ldots, x_{n-1})$ and for every $\gamma_0, \ldots, \gamma_{n-1} < \lambda$, there is some $\gamma$ such that for every $i \in I$, $M[i] \models \exists x \phi(x, a_{\gamma_0}(i), \ldots, a_{\gamma_{n-1}}(i))$ iff $M[i] \models \phi(a(i), a_{\gamma_0}(i), \ldots, a_{\gamma_{n-1}}[i])$. (This is an easy part of Lemma 6.1 from \cite{Optimals}.)

Let $\{\phi_\alpha(x, \overline{a}): \alpha < \lambda\}$ be an enumeration of all formulas $\phi_\alpha(x, \overline{a})$ with parameters from $\overline{a}$, such that $p(x) \models \phi_\alpha(x, ([[a_\gamma/\mathcal{U}_*]]: \gamma <\lambda))$. For each $s \in [\lambda]^{<\aleph_0}$ let $B_s = \{i \in I: M[i] \models \exists x \bigwedge_{\alpha \in s} \phi_\alpha(x, \overline{a}(i))\} \in \mathcal{U}_*$ and let $\mathbf{b}_s = \mathbf{j}(B_s) \in \mathcal{U}$. Now $p(x)$ is realized in $\overline{M}$ if and only if $(B_s)$ has a multiplicative refinement in $\mathcal{U}_*$; by (the proof of) Separation of Variables, this is the case if and only if $(\mathbf{b}_s)$ has a multiplicative refinement in $\mathcal{U}$. So it suffices to show that the last condition holds.

It is convenient to define a $\mathcal{B}$-valued model of $T$ with universe $\mathbf{M} = (a_\gamma: \gamma < \lambda)$, and with evaluation function $||\phi(\overline{a})|| = \mathbf{j}(\{i \in I: M[i] \models \phi(\overline{a}(i))\})$. (Here $\phi(\overline{a})$ is any formula with parameters from $\overline{a}$.) Note then that each $\mathbf{b}_s =  ||\exists x \bigwedge_{\alpha \in s} \phi_\alpha(x, \overline{a})||$.

Choose $\mbox{cl}$ a $\sigma$-closure relation on $\lambda$ such that:

\begin{itemize}
\item For each $w \subseteq \lambda$ $\mbox{cl}$-closed, $\overline{a}_w \preceq \mathbf{M}$ as $\mathcal{B}$-valued models; that is whenever $\gamma_0, \ldots, \gamma_{n-1} \in w$, and for every formula $\phi(x_0, \ldots, x_{n-1})$, there is $\gamma \in w$ with $||\exists x \phi(x, a_{\gamma_0}, \ldots, a_{\gamma_{n-1}})|| = ||\phi(a_\gamma, a_{\gamma_0}, \ldots, a_{\gamma_{n-1}})||$;
\item For each $w \subseteq \lambda$ $\mbox{cl}$-closed, $\{\phi_\alpha(x, \overline{a}): \alpha \in w\}$ enumerates a complete type over $\overline{a}_w$;
\item $p(x)$ does not fork over $\overline{a}_{\mbox{cl}(\emptyset)}$;
\item $T$ has $(\lambda, \mu, \sigma, \mbox{cl})$-type amalgamation.
\end{itemize}

This is clearly possible since the first three items are preserved under expansions. Let $G: \mathcal{R}(\mbox{cl}) \to P_{\lambda, \mu, \sigma}$ witness that $T$ has $(\lambda, \mu, \sigma, \mbox{cl})$-type amalgamation. It suffices to show that for every $\delta$ large enough so that $\mathbf{M}$ is a $\mathcal{B}_\delta$-valued model (that is, $||\cdot||$ takes its values in $\mathcal{B}_\delta$), we have that $(\mathbf{b}_s)$ has a multiplicative refinement consisting of nonzero elements mod $\mathcal{U} \restriction_{\mathcal{B}_\delta}$.

\vspace{2 mm}

\noindent \textbf{Interlude.} We now analyze some of the ways an element $\mathbf{a} \in \mathcal{B}_\delta$ can interact with a set $w \in [\lambda]^{<\sigma}$. Let $\mathcal{F}_0$ be the set of all pairs $(\mathbf{a}, w)$, where $w \in [\lambda]^{<\sigma}$ and $\mathbf{a} \in \mathcal{B}_\delta$ is nonzero. We put an order on $\mathcal{F}_0$ as follows: $(\mathbf{a}, w) \leq (\mathbf{a}',w' )$ if and only if $w \supseteq w'$ and $\mathbf{a} \leq \mathbf{a}'$. 

Let $\mathcal{F}_1$ be the set of all $(\mathbf{a}, w) \in \mathcal{F}_0$ such that whenever $\phi(\overline{a}_w)$ is a formula with parameters from $\overline{a}_w$, then $\mathbf{a}$ decides $||\phi(\overline{a}_w)||$ (i.e. either $\mathbf{a} \leq ||\phi(\overline{a}_w)||$ or else $\mathbf{a} \leq \lnot ||\phi(\overline{a}_2)||$). In this case $\mathbf{a}$ defines a complete type $q_{\mathbf{a}, w}(\overline{y}_w)$ in the variables $\overline{y}_w$, in the natural way.  

Since $P_{\delta, \mu, \sigma}$ is $\sigma$-closed, we have that $\mathcal{F}_1$ is dense in $\mathcal{F}_0$. Also if $(\mathbf{a}, w), (\mathbf{a}', w') \in \mathcal{F}_1$ and $\mathbf{a} \cap \mathbf{a}' \not= 0$, then $q_{\mathbf{a},w}$ and $q_{\mathbf{a}', w'}$ agree on $\overline{y}_{w \cap w'}$. Note also that if $(\mathbf{a}, w) \in \mathcal{F}_0$ and $w$ is $\mbox{cl}$-closed and $\overline{b}$ is a realization of $q_{\mathbf{a}, w}(\overline{y}_w)$ (from $\mathfrak{C}_T$) then $\overline{b}$ enumerates an elementary substructure of $\mathfrak{C}_T$.

Now, suppose $(\mathbf{a}, w) \in \mathcal{F}_1$, $\alpha \in w$ and $u \subseteq w \cap \alpha$. Then say that $u$ \emph{works} for $(\mathbf{a}, w, \alpha)$ if whenever $(\mathbf{a}', w') \leq (\mathbf{a}, w)$ is in $\mathcal{F}_1$ and whenever $\overline{b}$ is a realization of $q_{\mathbf{a}', w'}(\overline{y}_{w'})$ (so $\overline{b}_{w}$ realizes $q_{\mathbf{a}, w}(\overline{y}_w)$), then we have that $tp(b_\alpha / \overline{b}_{w' \cap \alpha})$ does not fork over $\overline{b}_{u}$. 

For example, if there is some $u \subseteq w \cap \alpha$ that works for $(\mathbf{a}, w, \alpha)$, then in particular $w \cap \alpha$ works for $(\mathbf{a}, w, \alpha)$.

\vspace{1 mm}

\noindent \textbf{Claim 1.} Suppose $(\mathbf{a}, w) \in \mathcal{F}_0$ and $\alpha \in w$. Then there is $(\mathbf{a}', w') \leq (\mathbf{a}, w)$ in $\mathcal{F}_1$ such that $w' \cap \alpha$ works for $(\mathbf{a}', w', \alpha)$.
\begin{proof}
Suppose not. Build $(f_\gamma, w_\gamma: \gamma < \omega_1)$ so that:

\begin{itemize}
\item $f_\gamma \in P_{\delta, \mu, \sigma}$ and $w_\gamma \in [\lambda]^{<\sigma}$;
\item $\mathbf{x}_{f_\gamma} \leq \mathbf{a}$, $w_0 = w$;
\item $\gamma \leq \gamma'$ implies $f_\gamma \subseteq f_{\gamma'}$ and $w_\gamma \subseteq w_{\gamma'}$;
\item For limit $\gamma < \aleph_1$ we have $f_\gamma= \bigcup_{\gamma' < \gamma} f_{\gamma'}$ and $w_\gamma = \bigcup_{\gamma' < \gamma} w_{\gamma'}$;
\item Each $(\mathbf{x}_{f_\gamma}, w_\gamma) \in \mathcal{F}_1$;
\end{itemize}
(There is one more condition, but first note that $\bigcup_{\gamma < \aleph_1} q_{\mathbf{x}_{f_\gamma}, w_\gamma}(\overline{y}_{w_\gamma})$ is a complete type in the variables $\overline{y}_{w_*}$, where $w_* = \bigcup_{\gamma < \sigma} w_\alpha$. Let $\overline{b}$ be a realization of this type; so $\overline{b}_{w_\gamma}$ realizes $q_{\mathbf{x}_{f_\gamma}, w_\gamma}$.)
\begin{itemize}
\item For every $\gamma < \sigma$, $tp(b_{\alpha}/\overline{b}_{w_{\gamma+1} \cap \alpha})$ forks over $\overline{b}_{w_\gamma \cap \alpha}$.
\end{itemize}
This is possible since if at some stage $\gamma$ we couldn't continue then clearly $(\mathbf{x}_{f_\gamma}, w_\gamma)$ would be as desired. But now we have contradicted the simplicity of $T$, since $tp(b_\alpha/\overline{b})$ forks over every countable subset of $\overline{b}$.
\end{proof}

Fix for the rest of the proof a well-ordering $<_*$ of $[\lambda]^{<\sigma}$. We use $<_*$ to perform collision detections similarly to the various saturation arguments in \cite{DividingLine}, \cite{Optimals}, \cite{InfManyClass}.

Given $(\mathbf{a}, w) \in \mathcal{F}_1$ and $\alpha \in w$, say that $u \in [\alpha]^{<\sigma}$ is a \emph{candidate} for $(\mathbf{a}, w, \alpha)$ if there is some $(\mathbf{a}', w') \leq (\mathbf{a}, w)$ in $\mathcal{F}_1$ such that $u \subseteq w' \cap \alpha$ and $u$ works for $(\mathbf{a}, w, \alpha)$.  Define $\pi_{\mathbf{a}, w}(\alpha)$ to be the $<_*$-least $u \in [\alpha]^{<\sigma}$ such that $u$ is a candidate for $(\mathbf{a}, w, \alpha)$. 

Let $\mathcal{F}_*$ be the set of all $(\mathbf{a}, w)$ in $\mathcal{F}_1$ such that $w$ is $\mbox{cl}$-closed, and for each $\alpha \in w$, $\pi_{\mathbf{a}, w}(\alpha) \subseteq w$ and further $\pi_{\mathbf{a}, w}(\alpha)$ works for $(\mathbf{a}, w, \alpha)$. Note that whenever $(\mathbf{a}, w), (\mathbf{a}', w') \in \mathcal{F}_2$ and $\mathbf{a} \land \mathbf{a}' \not= 0$, then for every $\alpha \in w \cap w'$ we have that $\pi_{\mathbf{a}, w}(\alpha) = \pi_{\mathbf{a}', w'}(\alpha) \subseteq w \cap w'$. 

\vspace{1 mm}

\noindent \textbf{Claim 2.} $\mathcal{F}_*$ is dense in $\mathcal{F}_1$ (and hence in $\mathcal{F}_0$).
\begin{proof}
Just note that given $(\mathbf{a}, w) \in \mathcal{F}_1$ and $\alpha \in w$, if we let $u := \pi_{\mathbf{a}, w}(\alpha)$, and if we choose $(\mathbf{a}', w') \leq (\mathbf{a}, w)$ in $\mathcal{F}_1$ witnessing that $u$ is a candidate for $(\mathbf{a}, w, \alpha)$, then for any $(\mathbf{a}'', w'') \leq (\mathbf{a}', w')$ in $\mathcal{F}_1$, we have that $\pi_{\mathbf{a}'', w''}(\alpha) = u$ and $u$ works for $(\mathbf{a}'', w'', \alpha)$.
\end{proof}

With Claim 2 in hand, we can now proceed with the proof of the theorem.

For each $\alpha < \lambda$ let $v_\alpha \in [\lambda]^{<\aleph_0}$ be the set of all $\gamma < \lambda$ such that $a_\gamma$ occurs in $\phi_\alpha(\overline{a})$; given $s \in [\lambda]^{<\aleph_0}$ let $v_s = \bigcup_{\alpha \in s} v_\alpha$. Let $\overline{\phi}_s(x, \overline{a}_{v_s})$ denote the formula $\bigwedge_{\alpha \in s} \phi_\alpha(x, \overline{a}_{v_\alpha})$. Now we know that $\overline{\phi}_s(x, \overline{a}_{v_s})$ does not fork over $\overline{a}_{\mbox{cl}(\emptyset)}$. Hence by Theorem~\ref{LowEquiv}(B) we can choose a formula $\psi_s(\overline{y}_{v_s}, \overline{y}_{\mbox{cl}(\emptyset)}) \in tp(\overline{a}_{v_s}, \overline{a}_{\mbox{cl}(\emptyset)})$, such that whenever whenever $\overline{b}_{v_s}, \overline{b}_{\mbox{cl}(\emptyset)} \in \mathfrak{C}$ are such that $\models \psi_s(\overline{b}_{v_s}, \overline{b}_{\mbox{cl}(\emptyset)})$, then $\overline{\phi}_s(x, \overline{b}_{v_s})$ does not fork over $\overline{b}_{\mbox{cl}(\emptyset)}$. (This is where we use $T$ is low; if instead $T$ were just simple as in Theorem 7.3 from \cite{Optimals} then we would need countably many formulas from $tp(\overline{a}_{v_s},\overline{a}_{\mbox{cl}(\emptyset)})$, and hence would need $\aleph_1$-completeness of $\mathcal{U}$.) Thus, whenever $(\mathbf{a}, w) \in \mathcal{F}_*$ is such that $v_s \subseteq w$ and $\mathbf{a} \leq \mathbf{b}_s \land ||\psi_s(\overline{a}_{v_s}, \overline{a}_{\mbox{cl}(\emptyset)})||$, and whenever $\overline{b}$ realizes $q_{\mathbf{a}, w}$, then $\overline{\phi}_s(x, \overline{b}_{v_s})$ does not fork over $\overline{b}_{\mbox{cl}(\emptyset)}$. Write $\mathbf{b}^*_s = \mathbf{b}_s \land ||\psi_s(\overline{a}_{v_s}, \overline{a}_{\mbox{cl}(\emptyset)})||$; so $\mathbf{b}^*_s \in \mathcal{U}$.

For each $s \in [\lambda]^{<\aleph_0}$, choose a sequence $((\mathbf{b}_{s, \xi}, w_{s, \xi}): \xi < \xi(s))$ such that:

\begin{itemize}
\item $(\mathbf{b}_{s, \xi}: \xi < \xi(s))$ is a maximal antichain of $\mathcal{B}_\delta$ (and hence $\mathcal{B}$) below $\mathbf{b}_s^*$;
\item each $(\mathbf{b}_{s, \xi}, w_{s, \xi})  \in \mathcal{F}_*$;
\item each $v_s \subseteq w_{s, \xi}$.

\end{itemize}

In particular, for each $s, \xi$, $q_{s, \xi}(\overline{y}_{w_{s, \xi}}) := q_{\mathbf{b}_{s, \xi}, w_{s, \xi}}(\overline{y}_{w_{s, \xi}})$ makes sense and $\pi_{s, \xi} := \pi_{\mathbf{b}_{s, \xi}, w_{s, \xi}}$ makes sense.

Now let $s \in [\lambda]^{<\sigma}$ and let $\xi < \xi(s)$. Then we can choose a complete type $p_{s, \xi}(x, \overline{y}_{w_{s, \xi}})$ in the listed variables so that for each $\alpha \in s$, $\phi_\alpha(x, \overline{y}_{v_\alpha}) \in p_{s, \xi}$, and such that furthermore if $\overline{b}$ is some or any realization of $q_{s, \xi}$ then $p_{s, \xi}(x, \overline{b})$ does not fork over $\overline{b}_{\mbox{cl}(\emptyset)}$. This is possible by choice of $\psi_s$.

Note then that clearly for each $s, \xi$ we have that $(w_{s, \xi}, q_{s, \xi}, p_{s, \xi}) \in \mathcal{R}(\mbox{cl})$, so it makes sense to consider $G(w_{s, \xi}, q_{s, \xi}, p_{s, \xi})) \in P_{\lambda, \mu, \sigma}$.

For each $s \in [\lambda]^{<\aleph_0}$ and for each $\xi < \mu$ let $h_{s, \xi} \in P_{2^\lambda \backslash \delta, \mu, \sigma}$ be such that, whenever $h_{s, \xi}$ and $h_{s', \xi'}$ are compatible, then:

\begin{itemize}
\item $p_{s, \xi}$ and $p_{s', \xi'}$ are compatible;
\item $G(w_{s, \xi}, q_{s, \xi}, p_{s, \xi})$ and $G(w_{s', \xi'}, q_{s', \xi'}, p_{s', \xi'})$ are compatible.
\end{itemize}

This is not hard to do (just shift everything over by multiples of $\delta$).

Finally, let $\mathbf{b}'_{\{\alpha\}} = \bigvee\{\mathbf{b}_{s, \xi} \land \mathbf{x}_{h_{s, \xi}}: \alpha \in s \in [\lambda]^{<\aleph_0}, \xi < \xi(s)\}$ and let $\mathbf{b}'_s = \bigwedge_{\alpha \in s} \mathbf{b}'_{\{\alpha\}}$. We claim that $(\mathbf{b}'_s)$ is a multiplicative refinement of $(\mathbf{b}_s)$ consisting of nonzero elements of $\mathcal{U} \restriction_{\mathcal{B}_\delta}$. Multiplicativity is clear. Also, let $s \in [\lambda]^{<\sigma}$; we show that $\mathbf{b}'_s$ is nonzero mod $\mathcal{U} \restriction_{\mathcal{B}_\delta}$. Indeed, suppose $\mathbf{a} \in \mathcal{U} \cap \mathcal{B}_\delta$; it suffices to show that $\mathbf{a} \land \mathbf{b}'_s$ is nonzero. We can suppose $\mathbf{a} \leq \mathbf{b}^*_s$. Since $(\mathbf{b}_{s, \xi}: \xi < \xi(s))$ is a maximal antichain below $\mathbf{b}^*_s$, there must be some $\xi < \xi(s)$ such that $\mathbf{a} \land \mathbf{b}_{s, \xi}$ is nonzero. Then $\mathbf{a} \land \mathbf{b}_{s, \xi} \land \mathbf{x}_{h_{s, \xi}}$ is nonzero, but $\mathbf{b}_{s, \xi} \land \mathbf{x}_{h_{s, \xi}} \leq \mathbf{b}'_s$, so we have shown $\mathbf{a} \land \mathbf{b}'_s$ is nonzero.

So it remains to show that each $\mathbf{b}'_s \leq \mathbf{b}_s$. Suppose not. Then we can choose $\mathbf{c} \leq -\mathbf{b}_s$ nonzero, such that for each $\alpha \in s$ there is $s_\alpha \in [\lambda]^{<\aleph_0}$ with $\alpha \in s$, and there is $\xi_\alpha < \xi(s_\alpha)$, such that $\mathbf{c} \leq \mathbf{b}_{s_\alpha, \xi_\alpha} \wedge \mathbf{x}_{h_{s_\alpha, \xi_\alpha}}$. Write $w_\alpha = w_{s_\alpha, \xi_\alpha}$, write $q_\alpha = q_{s_\alpha, \xi_\alpha}$, etc. Let $w = \bigcup_{\alpha \in s} w_\alpha$. Then we can decrease $\mathbf{c}$ so that $(\mathbf{c}, w) \in \mathcal{F}_1$; so $q_{\mathbf{c}, w}$ makes sense. Note that $q_{\mathbf{c}, w} = \bigcup_{\alpha \in s} q_\alpha$ and $\pi_{\mathbf{c}, w} = \bigcup_{\alpha \in s} \pi_{\alpha}$.

Let $\overline{b}$ be a realization of $q_{\mathbf{c}, w}$. Then we claim $\overline{b}$ is a good instantiation of $((w_\alpha, q_\alpha, p_\alpha): \alpha \in s)$. To verify the nonforking condition, choose $\gamma \in w$ and write $u = \pi_{\mathbf{c}, w}(\gamma)$. Then clearly $u \subseteq w_\alpha$ whenever $\gamma \in w_\alpha$, and $tp(b_\gamma/\overline{b}_{w \cap \gamma})$ does not fork over $\overline{b}_{u}$. The other conditions are also clear.

Now, $\bigcup_{\alpha \in s} G(w_\alpha, q_\alpha, p_\alpha)$ is a function since the $h_\alpha$'s are compatible, so that means $\bigcup_{\alpha \in s} p_\alpha(x, \overline{b}_{w_\alpha})$ is consistent (and does not fork over $\overline{b}_{\mbox{cl}(\emptyset)}$). But this implies that $\mathbf{c}$ is compatible with $\mathbf{b}_s$, a contradiction. 
\end{proof}

Our plan to show that low is a dividing line is to show that in the special case $\lambda = \mu^+$ of the preceding Theorem~\ref{SatArg}, we have $\mathcal{U}$ is $(\lambda, \mathcal{B}, T)$ moral if and only if $T$ is low. Theorem~\ref{SatArg} provides one direction, namely the saturation half of the argument; we now provide the other.

An ultrafilter $\mathcal{U}$ on the complete Boolean algebra $\mathcal{B}$ is $\lambda$-\emph{OK} if whenever $(\mathbf{a}_n: n < \omega)$ is a descending sequence from $\mathcal{U}$ with $\mathbf{a}_0 = 1$, then there is a multiplicative distribution $(\mathbf{b}_s: s \in [\lambda]^{<\omega})$ from $\mathcal{U}$ such that each $\mathbf{b}_s \leq \mathbf{a}_{|s|}$. The following is Conclusion 12.16 from \cite{pEqualsTref} by Malliaris and Shelah.

\begin{theorem}
Suppose $\mathcal{U}$ is a regular ultrafilter on $\mathcal{P}(\lambda)$. If $\mathcal{U}$ $\lambda^+$-saturates some nonlow theory then $\mathcal{U}$ is $\lambda$-OK.
\end{theorem}

The following is a straightforward generalization to complete Boolean algebras:

\begin{theorem}
Suppose $\mathcal{U}$ is an ultrafilter on a complete Boolean algebra $\mathcal{B}$. If $\mathcal{U}$ is $(\lambda, \mathcal{B}, T)$-moral for some nonlow theory $T$, then $\mathcal{U}$ is $\lambda$-OK.
\end{theorem}
\begin{proof}
By the Existence Theorem, we can find some index set $I$, some regular good filter $\mathcal{D}_0$ on $\mathcal{P}(I)$, and some Boolean algebra isomorphism $\mathbf{j}: \mathcal{P}(I)/\mathcal{D}_0 \cong \mathcal{B}$.  Let $\mathcal{U}_* =  \mathbf{j}^{-1}(\mathcal{U})$. Then $\mathcal{U}$ is $(\lambda, \mathcal{B}, T)$-moral if and only if $\mathcal{U}_*$ $\lambda^+$-saturates $T$ for every theory $T$, by Separation of Variables; so it suffices to show that if $\mathcal{U}_*$ is $\lambda$-OK, then so is $\mathcal{U}$.

So let $(\mathbf{a}_n: n < \omega)$ be a descending sequence from $\mathcal{U}$, and for $s \in [\lambda]^{<\aleph_0}$ define $\mathbf{a}_s = \mathbf{a}_{|s|}$. Choose $(A_n: n < \omega)$ a descending sequence from $\mathcal{P}(I)$ with $\mathbf{j}(A_n) = \mathbf{a}_n$ and define $A_s = A_{|s|}$ for $s \in [\lambda]^{<\aleph_0}$. Then since $\mathcal{U}_*$ is $\lambda$-OK, we have that $(A_s: s \in [\lambda]^{<\aleph_0})$ has a multiplicative refinement in $\mathcal{U}_*$; then the image of this multiplicative refinement under $\mathbf{j}$ witnesses that $\mathcal{U}$ is $\lambda$-OK.
\end{proof}

The following theorem generalizes Corollary 9.9 from \cite{DividingLine} (which is stated for certain special cases of $\mathcal{B}$); the proof is different, following Theorem 4.1 of \cite{BooleanUltrapowers}.

\begin{theorem}
Suppose $\mathcal{B}$ is a complete Boolean algebra with the $\lambda$-c.c. and $\mathcal{U}$ is an $\aleph_1$-incomplete ultrafilter on $\mathcal{B}$. Then $\mathcal{U}$ is not $\lambda$-OK; in particular $\mathcal{U}$ is not $(\lambda, \mathcal{B}, T)$-moral for any unsimple or nonlow theory.
\end{theorem}

\begin{proof}
By the previous two theorems, the second claim follows from the first. So it suffices to show $\mathcal{U}$ is not $\lambda$-OK.

Choose $(\mathbf{a}_n: n < \omega)$ a descending sequence from $\mathcal{U}$ such that $\mathbf{a}_0 = 1$ and $\bigwedge_n \mathbf{a}_n = 0$, and for $s \in [\lambda]^{<\aleph_0}$ define $\mathbf{a}_s = \mathbf{a}_{|s|}$. Suppose towards a contradiction that $(\mathbf{a}_s: s \in [\lambda]^{<\aleph_0})$ had a multiplicative refinement $(\mathbf{b}_s: s \in [\lambda]^{<\aleph_0})$.

We claim that for every $\mathbf{c} \in \mathcal{B}$ nonzero, there is $\mathbf{c}' \leq \mathbf{c}$ such that $\mathbf{c}'$ decides $\mathbf{b}_{\{\alpha\}}$ for each $\alpha < \lambda$.  Indeed, let $\mathbf{c} \in \mathcal{B}$ be nonzero. For each $n < \omega$ let $\mathbf{d}_n = \bigvee\{\mathbf{b}_s: s \in [\lambda]^n\}$. Then $(\mathbf{d}_n: n < \omega)$ is a descending sequence with $\mathbf{d}_0 = 1$, and its intersection is empty since each $\mathbf{d}_n \leq \mathbf{a}_n$. Thus there is some number $n$ with $\mathbf{c} - \mathbf{d}_n < \mathbf{c} - \mathbf{d}_{n+1}$. In other words, $\mathbf{c} \wedge (-\mathbf{d}_{n+1}) \wedge \mathbf{d}_n$ is nonzero. Thus there must be some $s_* \in [\lambda]^n$ with $\mathbf{c} \wedge (-\mathbf{d}_{n+1}) \wedge \mathbf{b}_{s_*}$ nonzero. Let $\mathbf{c}' = \mathbf{c} \wedge (-\mathbf{d}_{n+1}) \wedge \mathbf{b}_{s_*}$. Then we claim that $\mathbf{c}'$ decides each $\mathbf{b}_{\{\alpha\}}$. Indeed, for $\alpha \in s$ clearly $\mathbf{c}' \leq \mathbf{b}_{\{\alpha\}}$. On the other hand suppose $\alpha \not \in s$; then $\mathbf{c}' \wedge \mathbf{b}_{\{\alpha\}} = \mathbf{c}' \wedge \mathbf{b}_{s_* \cup \{\alpha\}} \leq \mathbf{c}' \wedge \mathbf{d}_{n+1} = 0$.

Thus we can choose a maximal antichain $C$ from $\mathcal{B}$ such that for every $\mathbf{c} \in C$, $\mathbf{c}$ decides every $\mathbf{b}_{\{\alpha\}}$. So each $\mathbf{b}_{\{\alpha\}}$ is supported on $C$, so we can choose $\mathbf{c}_\alpha \in C$ such that $\mathbf{c}_\alpha \leq \mathbf{b}_{\{\alpha\}}$. Then clearly $\alpha \mapsto \mathbf{c}_\alpha$ is a finite-to-one map from $\lambda$ to $C$, so $|C| \geq \lambda$, contradicting that $\mathcal{B}$ has the $\lambda$-c.c.
\end{proof}

Putting the previous theorems together we get:

\begin{theorem}
Suppose $\aleph_0 < \sigma \leq \mu = \mu^{<\sigma} \leq \lambda$ with $\sigma$ regular. Then there is an ultrafilter $\mathcal{U}$ on $\mathcal{B}_{2^\lambda, \mu, \sigma}$ such that for every countable complete theory $T$, if $T$ has $(\lambda, \mu, \sigma)$-type amalgamation and is low, then $\mathcal{U}$ is $(\lambda, \mathcal{B}, T)$-moral, but if $T$ is nonlow then $\mathcal{U}$ is not $(\mu^{++}, \mathcal{B}, T)$-moral. In particular, if $\lambda =\mu^+$, then $\mathcal{U}$ is $(\lambda, \mathcal{B}, T)$-moral iff $T$ is low.

Thus if $T$ is low and $T' \trianglelefteq T$ then $T'$ is low.
\end{theorem}
\begin{proof}
Let $\mathcal{U}$ be $\lambda$-perfect and $\aleph_1$-incomplete. Then this works by the preceding theorems.

To see that such cardinals $(\lambda, \mu, \sigma)$ exist, let $\sigma = \aleph_1$ and $\mu = 2^{\aleph_0}$ and $\lambda = \mu^+$, say.
\end{proof}

\section{A Minimal Nonlow Theory}\label{MinNonLowSec}

In this section, we show there is a minimal nonlow theory $T_{cas}$. This result is new, although the theory is not new, rather it is due to Casanovas \cite{Casanovas} and was in fact the first example of a simple nonlow theory. The language $\mathcal{L}_{cas}$ is $(R, P, I, I_n: n < \omega)$, where $P, I, I_n$ are each unary relation symbols and $R$ is binary. (Casanovas requires $n \geq 1$ but allowing $n = 0$ is harmless.) We adopt the convention that $a, a',\ldots$ are elements of $P$, $b, b', \ldots$ are elements of $I$.

Let $T_{cas}$ be the theory axiomatized by the following.

\begin{enumerate}
\item The universe is the disjoint union of $P$ and $I$, both infinite;
\item Each $I_n \subseteq I$, and the $I_n$'s are infinite and disjoint;
\item $R \subseteq P \times I$;
\item For each $a \in P$ and for each $n < \omega$, there are exactly $n$ elements $b \in I_n$ such that $R(a,b)$;
\item Whenever $B_0, B_1$ are finite disjoint subsets of $I$ such that each $|B_1 \cap I_n| \leq n$, there is $a \in P$ such that $R(a,b)$ for all $b \in B_1$ and $\lnot R(a, b)$ for all $b \in B_0$.
\item For all $A_0, A_1$ finite disjoint subsets of $P$, there is $b \in I$ such that $R(a, b)$ for all $a \in A_1$ and $\lnot R(a, b)$ for all $a  \in A_0$.
\end{enumerate}

In \cite{Casanovas} it is shown that $T_{cas}$ is complete, and is the model companion of the theory axiomatized by the first four items above. In particular, it is shown that $T_{cas}$ has quantifier elimination in an expanded language, where we add predicates $S_{\ldots}$ that express the following: given $A_0, A_1 \subset P$ finite disjoint with $A_0 \not= \emptyset$, how many $b \in I_n$ are there such that $R(a, b)$ for all $a \in A_1$ and $\lnot R(a, b)$ for all $a \in A_0$. Thus the algebraic closure of a set $X$ is $X \cup \bigcup\{b \in \bigcup_n I_n: \mbox{ there is } a \in X \cap P \mbox{ such that } R(a, b)\}$, and every formula over a set $X$ is equivalent to a quantifier-free formula over $\mbox{acl}(X)$.

Casanovas also shows that $T_{cas}$ is simple with the following forking relation: $X \forkindep_{Z} Y$ iff $\mbox{acl}(X) \cap \mbox{acl}(Y) \subseteq \mbox{acl}(Z)$. Also the formula $R(x, y)$ witnesses that $T_{cas}$ is not low.

The following lemma is immediate from the quantifier elimination in the expanded language discussed above:

\begin{lemma}\label{TCasLemma1}
Let $M \models T_{cas}$. As notation let $I_\omega$ denote $I \backslash \bigcup_n I_n$.
\begin{itemize} 
\item For each $n < \omega$, there is a unique nonalgebraic type $p(x)$ over $M$ with $I_{n}(x) \in p(x)$. It is isolated by the formulas $I_n(x)$ together with $\lnot R(a, x)$ for each $a \in P^M$.
\item For each $A \subseteq P^M$ let $p_A(x)$ be the type over $M$ that says $I_\omega(x)$ holds, $x \not= b$ for each $b \in I^M$, and finally for each $a \in P^M$, $R(a, x)$ holds iff $a \in A$. Then $p_A(x)$ generates a complete type over $M$ that does not fork over $\emptyset$. Moreover, all nonalgebraic complete types over $M$ extending $\{I(x)\} \cup \bigcup_n\{\lnot I_n(x)\}$ are of this form. 
\item Suppose $B \subseteq I^{M}$ is such that each $|B \cap I^{M}_n| \leq n$. Let $p_B(x)$ be the type over $M$ that says $P(x)$ holds, and $x \not= a$ for each $a \in P^M$, and for each $b \in I^{M}$, $R(x, b)$ holds iff $b \in B$. Then $p_B(x)$ generates a complete type over $M$, and moreover every complete nonalgebraic type over $M$ extending $P(x)$ is of this form. Further, given $M_0 \subseteq M$, we have that $p(x)$ does not fork over $M_0$ iff for each $n < \omega$, $B \cap I_n^{M_0} = B \cap I_n^M$.
\end{itemize}
\end{lemma}

From this lemma we get the following characterization of the saturated models of $T_{cas}$.

\begin{lemma}\label{TCasLemma1v2}
 $M \models T_{cas}$ is $\lambda^+$-saturated if and only if the following conditions are all satisfied:

\begin{itemize}
\item[(I)] $|I_\alpha| \geq \lambda^+$ for each $\alpha \leq \omega$;
\item[(II)] For all $B_0, B_1 \subseteq I^M$ disjoint with each $|B_i| \leq \lambda$, and with each $|B_1 \cap I_n| \leq n$, there is $a \in P$ such that $R(a, b)$ for each $b \in B_1$, and $\lnot R(a, b)$ for each $b \in B_0$; and
\item[(III)] For all $A_0, A_1 \subseteq P^M$ disjoint with each $|A_i| \leq \lambda$, there is $b \in I_\omega$ such that $R(a, b)$ for each $a \in A_1$ and $\lnot R(a, b)$ for each $a \in A_0$.
\end{itemize}
\end{lemma}

Before finishing we quote the following fact, which follows from the proof of Lemma 5.3 from \cite{RandGraph} showing that $T_{rg}$ is the minimal unstable theory; here $T_{rg}$ is the theory of the random graph. 

\begin{theorem}\label{RandGraphChar}
Let $\mathcal{U}$ be a regular ultrafilter on $\mathcal{P}(\lambda)$, and let $S$ be any infinite set. Then $\mathcal{U}$ $\lambda^+$-saturates $T_{rg}$ if and only if for any disjoint sets $A, B \subseteq S^\lambda/\mathcal{U}$, there is a sequence $(X_a: a \in A \cup B)$ with each $X_a \in \mathcal{U}$, such that for every $a \in A$ and for every $b \in B$, $X_a \cap X_b \subseteq \{\alpha < \lambda: a[\alpha] \not= b[\alpha]\}$. (Here $\alpha \mapsto a[\alpha]$ is some fixed lifting of $a$ to an element of $S^\lambda$.)
\end{theorem}

\begin{theorem}\label{MinNonLow}
Suppose $\mathcal{U}$ is a regular ultrafilter on $\mathcal{P}(\lambda)$. Then $\mathcal{U}$ $\lambda^+$-saturates $T_{cas}$ if and only if $\mathcal{U}$ $\lambda^+$-saturates $T_{rg}$ and $\mathcal{U}$ is $\lambda$-OK. Hence $T_{cas}$ is the minimal nonlow theory in Keisler's order.
\end{theorem}
\begin{proof}
If $T$ is any nonlow thoery and $\mathcal{U}$ $\lambda^+$-saturates $T$, then by Lemma 1.21 from \cite{MalliarisFlex} $\mathcal{U}$ is $\lambda$-OK; and since $T$ is unstable, by Lemma 5.3 from \cite{RandGraph}, $\mathcal{U}$ $\lambda^+$-saturates $T_{rg}$. So left to right is clear, as is the ``hence" statement. So to prove the theorem it suffices to show that if $\mathcal{U}$ $\lambda^+$-saturates $T_{rg}$ and $\mathcal{U}$ is $\lambda$-OK, then $\mathcal{U}$ $\lambda^+$-saturates $T_{cas}$. 

Let $M \models T_{cas}$; for convenience we suppose $I_\omega^M = \emptyset$.  Let $\overline{M} = M^I/\mathcal{U}$. Since $\mathcal{U}$ is $\lambda$-regular, clearly for each $\alpha \leq \omega$, $|U_\alpha^{\overline{M}}| \geq \lambda^+$, and since $\mathcal{U}$ additionally $\lambda^+$-saturates $T_{rg}$, by Theorem~\ref{RandGraphChar}, every partial type over $\overline{M}$ as in (III) from the previous lemma is realized. So it suffices to show every partial type over $\overline{M}$ as in (II) from the previous lemma is realized. 

So let $B_0, B_1 \subset I^{\overline{M}}$ be as in (II) from Lemma~\ref{TCasLemma1v2}, and let $p(x)$ be the partial type over $\overline{M}$ saying $x \in P$ and $R(x, b)$ for all $b \in B_1$ and $\lnot R(x, b)$ for all $b \in B_0$. Note that it clearly suffices to consider the case where for each $n < \omega$, $|B_1 \cap I_n^{\overline{M}}| = n$; after arranging this we can also suppose that $B_0 \subseteq I_\omega^{\overline{M}}$ since the other elements add no information. Let $(b_\alpha: \alpha < \lambda)$ enumerate $B_0 \cup B_1$. For each $\beta\leq \omega$ let $\Gamma_\beta = \{\alpha < \lambda: b_\alpha \in I_\beta^{\overline{M}}\}$; so $|\Gamma_n| = n$ for each $n < \omega$. Also, for $j = 0, 1$ let $\Gamma_{\omega, j} = \{\alpha < \lambda: b_\alpha \in B_j \cap I_\omega^{\overline{M}}\}$.

For each $\alpha < \lambda$, let $j_\alpha < 2$ be such that $b_\alpha \in B_{j_\alpha}$. Thus $p(x)$ can be written as $(R(x, b_\alpha)^{j_\alpha}: \alpha < \lambda)$. For each $s \in [\lambda]^{<\aleph_0}$ let $A_s$ be the set of all $i \in I$ such that $M[i] \models \exists x \bigwedge_{\alpha \in I} R(x, b_\alpha[i])^{j_\alpha}$. So each $A_s \in \mathcal{U}$; it suffices to find a multiplicative refinement to $(A_s)$ in $\mathcal{U}$.

For each $\alpha \in \Gamma_{\omega}$ let $F_\alpha: I \to \omega$ be defined by $F_\alpha(i) = $ that $n < \omega$ with $b_\alpha[i] \in I_n^{M[i]}$; so each $F_\alpha$ is a nonstandard element of $(\omega, <)^I/\mathcal{U}$. Since $\mathcal{U}$ $\lambda^+$-saturates $T_{rg}$, we can by Theorem 4.8 from Chapter 6 of \cite{ShelahIso} choose $F: I \to \omega$ $\mathcal{U}$-nonstandard such that $F \leq_{\mathcal{U}} F_\alpha$ for each $\alpha \in \Gamma_{\omega}$. For each $n < \omega$ let $C_n = \{i \in I: F(i) \geq n\} \in \mathcal{U}$. For $s \in [\lambda]^{<\aleph_0}$ let $C_s = C_{|s|}$. Since $\mathcal{U}$ is $\lambda$-OK we can choose a multiplicative refinement $(D_s: s \in [\lambda]^{<\aleph_0})$ of $(C_s)$ in $\mathcal{U}$.

Also, since $\mathcal{U}$ $\lambda^+$-saturates $T_{rg}$, we can by Theorem~\ref{RandGraphChar} choose $(E_\alpha: \alpha \in \Gamma_\omega)$ from $\mathcal{U}$ such that whenever $\alpha_j \in \Gamma_{\omega, j}$ for $j < 2$, and whenever $i \in E_{\alpha_0} \cap E_{\alpha_1}$, then $M[i] \models b_{\alpha_0}[i] \not= b_{\alpha_1}[i]$.

Now given $\alpha < \lambda$, we define $A'_{\{\alpha\}}$ as follows. First, if $\alpha \in \Gamma_n$ then define $A'_{\{\alpha\}} := C_{n+1} \cap \{i \in I: b_\alpha[i] \in I_n^{M[i]}\}$. Second, if $\alpha \in \Gamma_{\omega, 1}$ then let $A'_{\{\alpha\}} := E_\alpha \cap D_{\{\alpha\}} \cap \{i \in I: F(i) \leq F_\alpha(i)\}$. If $\alpha \in \Gamma_{\omega, 0}$ then let $A'_{\{\alpha\}} = E_\alpha \cap \{i \in I: F(i) \leq F_\alpha(i)\}$. For $s \in [\lambda]^{<\aleph_0}$ let $A'_s = \bigcap_{\alpha \in s} A'_{\{\alpha\}}$.

So $(A'_s)$ is a multiplicative distribution in $\mathcal{U}$; it suffices to show it refines $A_s$. So let $s$ be given and let $i \in A'_s$. For each $\beta \leq \omega$ let $s_\beta = s \cap \Gamma_\beta$ and for each $j < 2$ let $s_{\omega, j} = s \cap \Gamma_{\omega, j}$. We want to check that $M[i] \models \exists x \left( \bigwedge_{\alpha \in s \backslash s_{\omega, 0}} R(x, b_\alpha[i]) \, \land \, \bigwedge_{\alpha \in s_{\omega, 0}} \lnot R(x, b_\alpha[i])\right)$.

We have that for each $n < \omega$, $\{b_\alpha[i]: \alpha \in s_n\}  \subseteq I_n^{M[i]}$ has size at most $n$. Let $n_*$ be largest so that $s_{n_*}$ is nonempty. Then $i \in C_{n_*+1}$, so for each $\alpha \in s_{\omega}$, $F_\alpha(i) \geq n_*+1$, that is $b_\alpha[i] \in \bigcup_{n > n_*} I_n^{M[i]}$. Further, if $\alpha_i \in s_{\omega, j}$ for $j < 2$, then by choice of $(E_\alpha: \alpha \in \Gamma_\omega)$ we have that $b_{\alpha_0}[i] \not= b_{\alpha_1}[i]$.

Thus the only issue that can arise is that there is some $n > n_*$ such that $\{\alpha \in s_{\omega, 1}: F_\alpha(i) = n\}$ has more than $n$ elements. In particular this would imply that $s_{\omega, 1}$ has more than $n$ elements, but then $i \in C_{n+1}$ by choice of $(D_s)$, so for each $\alpha \in s_{\omega, 1}$, $F_\alpha(s) \geq n+1$, contradiction.

\end{proof}

\end{document}